\documentclass[10pt,a4paper]{article}
\usepackage[utf8]{inputenc}
\usepackage{amsmath,amsthm}
\usepackage{amsfonts}
\usepackage{amssymb}
\usepackage{graphicx}

\usepackage{hyperref}
\hypersetup{colorlinks=true, linkcolor=blue, urlcolor=blue, 
  pdftitle={}, pdfauthor={}}

\newtheorem{theorem}{Theorem}
\newtheorem{lemma}{Lemma}
\newcommand{\ia}{\mathfrak{a}}
\newtheorem{prop}{Proposition}

\newcommand{\re}{\mathrm{Re}}
\newcommand{\fa}{\mathfrak{a}}

\newcounter{remark}
\newenvironment{remark}[1][Remark \theremark]{\refstepcounter{remark} \begin{trivlist} \item[\hskip 
\labelsep{\bfseries #1.}]\setlength{\parindent}{0pt}}{\hfill\end{trivlist}}

\providecommand{\keywords}[1]
{
  \small	
  \textbf{Keywords:} #1
}

\providecommand{\class}[1]
{
  \small	
  \textbf{Subject Classification Codes:} #1
}

\title{An explicit upper bound for Siegel zeros of imaginary quadratic fields}
\author{D.~Ralaivaosaona  
   \and F.~B.~Razakarinoro}
   \newcommand{\Addresses}{{
  \bigskip
  \footnotesize

  D.~Ralaivaosaona, \textsc{Department of Mathematical Sciences, Stellenbosch University, South Africa}\par\nopagebreak
  \textit{E-mail address}, D.~Ralaivaosaona: \texttt{naina@sun.ac.za}

  \medskip

  F.~B.~Razakarinoro, \textsc{Department of Mathematical Sciences, Stellenbosch University, South Africa}\par\nopagebreak
  \textit{E-mail address}, F.~B.~Razakarinoro: \texttt{brice@aims.ac.za}

}}

\date{}
\begin{document}
\maketitle
\begin{abstract}
For any integer $d\geq 3$ such that $-d$ is a fundamental discriminant, we show that the Dirichlet $L$-function associated with the real primitive character $\chi(\cdot)=(\frac{-d}{\cdot})$ does not vanish on the positive part of the interval  
$[1-6.5/\sqrt{d},\ 1].
$
\end{abstract}


\class{11M20}
\section{Introduction}
For a fundamental discriminant $D$, the arithmetic function defined by the Kronecker symbol $\chi(n)=(\frac{D}{n})$ is a real primitive Dirichlet character and its associated $L$-function is defined by the series
\[
L(s,\chi):=\sum_{n=1}^{\infty}\frac{\chi(n)}{n^s}.
\]
The series on the right-hand side only makes sense when $\re(s)>1$, but it is well known that  the function $L(s,\chi)$ has an analytic continuation defined over the whole complex plane. The locations of the zeros of $L(s,\chi)$ are particularly important in number theory. One of the most important open problem in mathematics -- the Generalized Riemann Hypothesis (GRH) -- asserts that all zeros with positive real parts lie precisely on the vertical line $\re(s)=\tfrac{1}{2}.$ 

Siegel zeros or sometimes called Landau-Siegel zeros are hypothetical real zeros of the $L$-functions that lie very close to $1$. The existence of these zeros has not yet been ruled out, but it is known that $L(s,\chi)$ has at most one simple zero in an interval of the form $(1-c/\log |D|,\ 1)$ , see Page \cite{Page}. Morrill and Trudgian, in \cite{Morrill-Trudgian}, recently  gave an explicit version of the latter statement with $c=1.011$  using Pintz's refinement of Page's theorem. The largest positive zero of $L(s,\chi)$, if it exists, will be denoted by $\beta$ throughout this paper.

We are interested in the upper bound on $\beta$, or equivalently the lower bound on the distance from $\beta$ to $1$, for the case $D=-d$ where $d\geq 3$. It is well known that there exists an absolute constant $c>0$ such that $1-\beta > c/\sqrt{d}$, see Haneke \cite{Haneke}, Goldfeld and Schinzel \cite{Gold}, and Pintz \cite{Pintz}. In particular, it is shown in the Goldfeld-Schinzel's paper that
\begin{equation}\label{eq:betaG-S}
1-\beta > \left(\frac{6}{\pi}+o(1)\right)\frac{1}{\sqrt{d}} \ \ \text{as } \ \ d\to \infty.
\end{equation}
Pintz achieved a similar result, but with a different method. He improved the constant $\frac{6}{\pi}$ to $\frac{12}{\pi}$, and then improved it further to $\frac{16}{\pi}$ following Schinzel's remark, see the footnote on page 277 of \cite{Pintz}. We are unaware of any result of the form $1-\beta > c/\sqrt{d}$ with an explicit constant $c>0$ prior to this work.  Known explicit results have an additional $(\log d)^2$ term in the denominator, see \cite[Lemma 3]{Ford}, \cite{Bennett}, and \cite{Bordignon}. Most of these papers made use of  explicit upper bounds for $L'(1,\chi)$.  We do not follow this route, instead, we use the method of Goldfeld and Schinzel in \cite{Gold}.

It is worth noting that $L(s,\chi)$ does not have positive real zeros for at least a positive proportion of fundmental discriminats $-d$, see \cite{Conrey}. Moreover, Watkins' computational results in \cite{Watkins1} show that the same holds for all $L(s,\chi)$  with fundamental discriminants $-d$  such that $d\leq 300000000$. The following theorem is our main result.

\begin{theorem}\label{thm:main}
Let $d> 300000000$ such that $-d$ is a fundamental discriminant. Let $L(s,\chi)$ be the Dirichlet $L$-function associated with  the primitive character  $\chi(n)=\left(\frac{-d}{n}\right)$. If there exists $\beta>0$ such that $L(\beta,\chi)=0$, then 
\begin{equation}\label{eq:beta}
1-\beta >  \frac{6.5}{\sqrt{d}}.
\end{equation}
\end{theorem}
Another Watkins' paper \cite{Watkins2} provides a classification of all imaginary quadratic fields with class number less or equal to $100$.  The combination of the results from \cite{Watkins1} and \cite{Watkins2} guarantees that we may only consider the case where the class number $h(-d)$ of the corresponding imaginary quadratic field $\mathbb{Q}(\sqrt{-d})$ is at least $101$. We will see that a higher class number gives a better constant in \eqref{eq:beta}. In fact, we have the following asymptotic result in terms of the class number.
\begin{theorem}\label{thm:t2}
Let $d$ and $\beta$ be as in Theorem~\ref{thm:main} and let  $h(-d)$ be the class number of the quadratic field $\mathbb{Q}(\sqrt{-d}).$ Then, we have 
\begin{equation}\label{eq:beta2}
1-\beta > \Big(2\pi+o(1)\Big)\frac{h(-d)}{(\log h(-d))^2\, \sqrt{d}}\, \, \text{ as }\, \, h(-d)\to \infty. 
\end{equation}
\end{theorem}
It is also possible to obtain an explicit bound for the $o(1)$ term in \eqref{eq:beta2} in terms of $h(-d)$. But this will not be very useful unless we have an explicit lower bound for the class number $h(-d)$, which is a much harder problem.

This paper is organized as follows: In Section~\ref{sec-pre} we prove two preliminary results, one on the sum of reciprocal prime powers and the other on the sum of reciprocal ideal norms. The proof of Theorem~\ref{thm:main} and Theorem~\ref{thm:t2} are done in Section~\ref{sec-tmain} and Section~\ref{sec-t2} respectively. We conclude with a short discussion on possible improvements of Theorem~\ref{thm:main} in Section~\ref{sec-conc}.  

\section{Preliminary results}\label{sec-pre}

\subsection{Sum of reciprocal prime powers}
We are going to need explicit estimates for the sum $\sum_{p^{\alpha}\leq x }p^{-\alpha}$, where the sum is taken over the prime powers $p^{\alpha}$ not exceeding $x$. It is clear that $\sum_{p^{\alpha}\leq x }p^{-\alpha}$ is greater than the sum of reciprocal primes $\sum_{p\leq x}p^{-1}$ but they are asymptotically equal as $x\to\infty$. It is well known that for $x\geq 3$, we have
$$
\sum_{p\leq x}p^{-1}= \log \log x +B_1+o(1),
$$ 
where $B_1$ is known as the Mertens constant, ref. Sequence A077761 in the OEIS. Dusart \cite{Dusart} recently provided an  explicit bound for the error term in the above estimate. It is shown, see \cite[Theorem 5.6]{Dusart}, that for every $x\geq 2278383$, we have
\begin{equation}\label{dusart}
\Big|\sum_{p\leq x}p^{-1}-\log\log x-B_1\Big|\leq \frac{0.2}{(\log x)^3}.
\end{equation}
We use this result to obtain an explicit estimate for the sum of reciprocal prime-powers.
\begin{prop}\label{prop}
For every $x\geq 2$, we have 
\begin{equation}\label{eq:recip}
-\frac{1.75}{(\log x)^2}\leq \sum_{p^{\alpha}\leq x}p^{-\alpha}-\log\log x-B_2\leq \min \left\lbrace\frac{0.2}{(\log x)^3},\ 10^{-4}\right\rbrace,
\end{equation}
where  
\[
B_2=B_1+\sum_{\alpha\geq 2}\sum_{p}p^{-\alpha}= 1.03465\ldots
\]
\end{prop}
The constant $B_1$ is sometimes referred to as the prime reciprocal constant, so we could  analogously call $B_2$ as the prime power reciprocal constant. $B_2$ also appears in the OEIS as Sequence A083342.

\begin{remark} The lower bound in \eqref{eq:recip} could be made of the form $c/(\log x)^3$ as in Dusart's result above, but we chose to use the asymptotically weaker bound in the proposition as it gives a slightly better approximation for small values of $x$, see the comparison with the exact error in Figure~\ref{fig}.
\begin{center}
\begin{figure}[h]\label{fig}
\centering
\includegraphics[scale=.6]{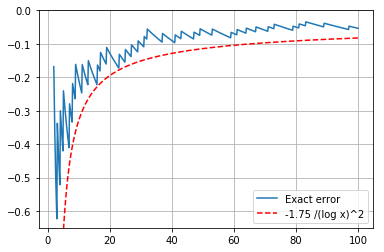}
\caption{Lower bound in \eqref{eq:recip} and exact error.}
\end{figure} 
\end{center}

\end{remark}
\begin{proof}
For $x\geq 2$, we have 
\[
0\leq \sum_{p^{\alpha}\leq x}p^{-\alpha} -\sum_{p\leq x}p^{-1}\leq \sum_{\alpha\geq 2}\sum_{p}p^{-\alpha}.
\]
Let us denote the double summation on the right-hand side by $C.$ It is easy to check that it is convergent. In fact, one has
\[
C=\lim_{N\to\infty}\sum_{\alpha\geq 2}\sum_{p\leq N}p^{-\alpha}=\lim_{N\to\infty}\sum_{p\leq N}\sum_{\alpha\geq 2}p^{-\alpha}=\sum_{p}\frac{1}{p^2-p}.
\]
This implies that 
\begin{equation}\label{upper}
\sum_{p^{\alpha}\leq x}p^{-\alpha} \leq \sum_{p\leq x}p^{-1}+C.
\end{equation}

Now, for the lower bound, we have 
\[
\sum_{p^{\alpha}\leq x}p^{-\alpha} -\sum_{p\leq x}p^{-1} 
= C-\sum_{\alpha\geq 2}\sum_{p^{\alpha}>x}p^{-\alpha},
\]
and
\[
\sum_{\alpha\geq 2}\sum_{p^{\alpha}>x}p^{-\alpha}\leq \sum_{p> \sqrt{x}}\frac{1}{p^2-p}\leq \sum_{n>\sqrt{x}}\frac{1}{n^2-n}=\frac{1}{\lceil \sqrt{x}\rceil-1}.
\]
Thus 
\begin{equation}\label{lower}
\sum_{p^{\alpha}\leq x}p^{-\alpha} \geq \sum_{p\leq x}p^{-1}+C-\frac{1}{\lceil \sqrt{x}\rceil-1}.
\end{equation}
Combining \eqref{upper} and \eqref{lower}, with Dusart's bound \eqref{dusart}, we obtain the following: for $x\geq 2278383$, 
\begin{equation}\label{bound2}
-\frac{0.2}{(\log x)^3}-\frac{1}{\lceil \sqrt{x}\rceil-1}
\leq \sum_{p^{\alpha}\leq x}p^{-\alpha}-\log\log x -B_1-C\leq \frac{0.2}{(\log x)^3}.
\end{equation}
It is easy to check that for $x\geq 2278383$, the latter bounds imply the estimates \eqref{eq:recip} in the statement of the proposition (if $x\geq 2278383$, then $0.2/(\log x)^3<10^{-4}$). It now remains to check that \eqref{eq:recip} also holds for all $x<2278383.$ We can use a computer check for this, but we need to be cautious because $x$ can take any real value. First, for $x\geq 2$, we let 
\[
\varepsilon(x):=\sum_{p^{\alpha}\leq x}p^{-\alpha}-\log\log x -B_2.
\]    
Then, we can easily show from this definition that if $p^{\alpha}$ is the greatest prime power $\leq x$ then
\[
\varepsilon(x)\leq  \varepsilon(p^{\alpha}).   
\] 
We verified numerically with a computer program that $\varepsilon(p^{\alpha})<0$ for all prime powers $p^{\alpha}$ in the interval $[2,\  2278383]$, which proves the upper bound in \eqref{eq:recip}. Similarly, if $p^{\alpha}$ is the least prime power $\geq x$ , then we have 
\[
\varepsilon(x)+\frac{1.75}{(\log x)^2}\geq \varepsilon(p^{\alpha})+\frac{1.75}{(\log p^{\alpha})^2}-
\begin{cases}
p^{-\alpha} & \text{if }  x<p^{\alpha}\\[.5em]
0 & \text{if }  x=p^{\alpha}. 
\end{cases}
\]
Again, we checked with a computer program that $\varepsilon(p^{\alpha})+\frac{1.75}{(\log p^{\alpha})^2}-p^{\alpha}>0$ for all prime powers $p^{\alpha}$ in the interval $[2,\  2278421]$ (the number 2278421 is the smallest prime power greater than 2278383). Therefore, we deduce that $\varepsilon(x)+\frac{1.75}{(\log x)^2}> 0$ for all $x\in [2,\  2278383]$, which completes the proof of the proposition.
\end{proof}

\subsection{Exploiting the class number} \label{hd}
The approach of Goldfeld and Schinzel involves reciprocal sums of norms of ideals of the form   
\begin{equation}\label{eq:sum-norm}
\sum_{N(\ia)\leq x}\frac{1}{N(\ia)},
\end{equation}
where $x\geq 1$ and $\ia$ runs over all nonzero ideals of the ring of integers $\mathcal{O}_{\mathbb{Q}(\sqrt{-d})}$. In order to understand such sums, let us recall some useful results from the classical theory of imaginary quadratic fields. For each positive integer $a$ let $\nu(a)$ denote the number of representations of $a$ as a norm of an ideal of $\mathcal{O}_{\mathbb{Q}(\sqrt{-d})}$ that is not divisible by any rational integer $>1$. Such an ideal can be written uniquely in the form  
\[
\left[a,\tfrac{b+\sqrt{-d}}{2}\right]:=\left\lbrace an+\tfrac{b+\sqrt{-d}}{2} m\ :\ n,m\in \mathbb{Z}\right\rbrace,
\]
where $a\geq 1$,  $-a< b \leq a$, and $b^2\equiv -d\ (\mathrm{mod}~4a).$ Moreover, every other ideal can be written in the form $u[a,\tfrac{b+\sqrt{-d}}{2}]$, where $u$ is a positive integer, and the norm of such an ideal is $u^2a$. So we can rewrite the sum in \eqref{eq:sum-norm} as follows: 
\[
\sum_{N(\ia)\leq x}\frac{1}{N(\ia)}=\sum_{u^2a\leq x} \frac{\nu(a)}{u^2a}.
\]
We have the following important lemma concerning the arithmetic function $\nu(\cdot).$ It was given without proof in \cite{Gold}, so we will provide a quick proof here.
\begin{lemma}\label{lem:nu}
The function $\nu(\cdot)$ is multiplicative with 
\[
\nu(p^{\alpha})=
\begin{cases}
1+\chi(p) & \text{if}\ p \nmid d \text{ or }\ \alpha=1 \\
0 & \text{otherwise}.
\end{cases}
\]
\end{lemma}
\begin{proof}
The multiplicativity of $\nu(\cdot)$ follows easily from the unique factorization property of the ideals of $\mathcal{O}_{\mathbb{Q}(\sqrt{-d})}$.  As for the formula for $\nu(p^{\alpha})$, we use the charicterization of prime ideals in $\mathcal{O}_{\mathbb{Q}(\sqrt{-d})}$:
\begin{itemize}
\item If $\chi(p)=0$ and $\alpha=1$, then the only ideal with norm $p$ is the ideal $\mathfrak{p}$ with $(p)=\mathfrak{p}^2$.
\item If $\chi(p)=1$, then we have the factorization $(p)=\mathfrak{p_1}\mathfrak{p_2}$ with $N(\mathfrak{p_1})=N(\mathfrak{p_2})=p$. Hence, we have $(p^{\alpha})=\mathfrak{p}_1^{\alpha}\mathfrak{p}_2^{\alpha}.$ Thus, the only ideals with norm $p^{\alpha}$ that are not divisible by rational integers are  
$\mathfrak{p}_1^{\alpha}$ and $\mathfrak{p}_2^{\alpha}$, since any other choice will have to be divisible by both $\mathfrak{p}_1$ and $\mathfrak{p}_2$, i.e., divisible by $(p).$ 
\item If $\chi(p)=-1$, then $(p)$ is a prime ideal with norm $p^2.$ Hence, there are no ideals with norm $p^{\alpha}$ if $\alpha$ is odd. But if $\alpha$ is even, then any ideal with norm $p^{\alpha}$ will be divisible by $(p).$
\end{itemize}
The only remaining case is when $\chi(p)=0$ and $\alpha\geq 2$. However, since $-d$ is a fundamental discriminant, the only possibility for this to happen is for $d$ to be divisible by $4$,  $p=2$, and $\alpha = 2$ or $3$. But again, in this case, the only ideal with norm $2^{\alpha}$ is the ideal $\mathfrak{p}^{\alpha}$, where $\mathfrak{p}^2=(2).$ Since $\alpha\geq 2,$ such an ideal is divisible by $(2).$ 
\end{proof}
When studying sums over norms of ideals like \eqref{eq:sum-norm}, it is often useful to consider the Dedekind zeta function for $\mathbb{Q}(\sqrt{-d}).$ Let
\[
\zeta_{-d}(s):=\sum_{\ia}\frac{1}{N(\ia)^s}, 
\]
where $\ia$ runs over all nonzero ideals of $\mathcal{O}_{\mathbb{Q}(\sqrt{-d})}$ and $\re(s)>1$. Lemma~\ref{lem:nu} implies that   
\begin{equation}\label{eq:dedk}
\zeta_{-d}(s)=\zeta(s)L(s,\chi),
\end{equation}
which immediately provides an analytic continuation for $\zeta_{-d}(s)$. Equation~\eqref{eq:dedk} is well known, but also follows easily from Lemma~\ref{lem:nu}. Indeed, for $\re(s)>1$, we have 
\begin{align*}
\sum_{\ia}\frac{1}{N(\ia)^s}
& = \sum_{u^2a}\frac{\nu(a)}{u^{2s}a^s}\\
& =\sum_{u^2}\frac{1}{u^{2s}}\sum_{a}\frac{\nu(a)}{a^s}\\
& = \prod_{p}\left(1-p^{-2s}\right)^{-1}\prod_{p,\chi(p)=0}\left(1+p^{-s}\right)\prod_{p,\chi(p)=1}\left(1+2\frac{p^{-s}}{1-p^{-s}}\right)\\
& = \prod_{p}\left(1-p^{-s}\right)^{-1} \prod_{p,\chi(p)=-1}\left(1+p^{-s}\right)^{-1}\prod_{p,\chi(p)=1}\left(1-p^{-s}\right)^{-1}\\
& =\zeta(s)L(s,\chi).
\end{align*}

Every ideal $[a,\tfrac{b+\sqrt{-d}}{2}]$ corresponds to a binary quadratic form $ax^2+bxy+cy^2$, where $d=4ac-b^2$. Such a form is called reduced if $-a<b\leq a<c$ or $0\leq b\leq a=c$. The number of reduced forms is known as the class number of $\mathbb{Q}(\sqrt{-d})$, and we denote it by $h(-d)$. Watkins in \cite{Watkins2} gives all negative fundamental discriminant with class number less or equal to $100$. The largest absolute value of such discriminants is $2383747$ (whose class number is $98$). Moreover, it is shown in another Watkins' paper \cite{Watkins1} that for $d\leq 300000000$, the function $L(s,\chi)$ does not have positive real zeros. Hence, we can assume from now on that $d>300000000,$ and so 
\begin{equation}
h(-d)\geq 101.
\end{equation}
\begin{lemma}\label{lem:h}
Let $h(-d)$ be the class number of a quadratic field of discriminant $-d$ with $d>300000000$. Then, we have 
\begin{equation}
\sum_{a\leq \tfrac{1}{2}\sqrt{d}}\frac{\nu(a)}{a}\leq \frac{h(-d)}{11}.
\end{equation}
\end{lemma}

\begin{proof}
Notice first that for an ideal $[a, \frac{b+\sqrt{-d}}{2}]$ with norm $a\leq \tfrac{1}{2}\sqrt{d}$, the corresponding quadratic form $ax^2+bxy+cy^2$ is reduced. To see this, note that for $d>4$,  equality cannot hold for $a\leq \tfrac{1}{2}\sqrt{d}$ since, otherwise, $d/4$ would not be squarefree. So
\[
4a^2< d=4ac-b^2\leq 4ac,
\]
which yields $a< c.$ The above observation implies that each ideal class of $\mathbb{Q}(\sqrt{-d})$ contains at most one ideal of the form  $[a, \frac{b+\sqrt{-d}}{2}]$ with norm $a\leq \tfrac{1}{2}\sqrt{d}$, and in particular, we have  
\[
\sum_{a\leq \tfrac{1}{2}\sqrt{d}}\nu(a)\leq h(-d).
\]

On the other hand, using Lemma~\ref{lem:nu}, we can show that $\nu(a)\leq 2^{w(a)}$, where $w(n)$ denotes the number of distinct prime divisors of $n$, with $w(1)=0.$ Hence, we have
\[
\sum_{a\leq \tfrac{1}{2}\sqrt{d}}\frac{\nu(a)}{a}\leq \sum_{a\leq \tfrac{1}{2}\sqrt{d}}\frac{2^{w(a)}}{a}.
\]
One can simply verify with a calculator that 
$
\sum_{n=1}^{34}2^{w(n)}=101.
$
This implies that there can only be at most $101$ ideals of the form $[a, \frac{b+\sqrt{-d}}{2}]$ with norm $a$ less or equal to $34$. But since in our case, the class number $h(-d)$ is at least $101$, we may write
\[
\sum_{a\leq \tfrac{1}{2}\sqrt{d}}\frac{\nu(a)}{a}\leq  \sum_{n=1}^{34}\frac{2^{w(n)}}{n}+\frac{h(-d)-101}{35}.
\]
This is because the sum is larger if more small numbers $a$ are represented as norms of ideals. So in the above, $101$ ideals have norms from $1$ to $34$ and the norms of the rest must be at least $35.$ Hence, by evaluating the sum on the right in the above, we obtain
\[
\sum_{a\leq \tfrac{1}{2}\sqrt{d}}\frac{\nu(a)}{a}\leq 9.161+\frac{h(-d)-101}{35}\leq \frac{h(-d)}{11},
\]
for $h(-d)\geq 101$.
\end{proof}

\section{Proof of Theorem \ref{thm:main}}\label{sec-tmain}
The proof relies on estimates of sums of the form \eqref{eq:sum-norm} when $x$ is slightly larger than $\frac{1}{2}\sqrt{d}$. To make this precise, we consider an auxiliary function $f(d)\geq 1$, to be specified later. We set 
\[
x=\frac{1}{2}\sqrt{d}f(d).
\]  
From now on, we may assume that there exists $\beta>0$ such that $L(\beta,\chi)=0$ and that 
\begin{equation}\label{eq:beta_ass}
1-\beta\leq  \frac{6.5}{\sqrt{d}},
\end{equation}
for otherwise, there will be nothing to prove. Then, we define the integral
\begin{equation}\label{intI}
I := \frac{1}{2\pi i}\int_{2-i\infty}^{2+i \infty} \zeta_{-d}(s+\beta)\frac{x^s}{s(s+2)(s+3)}\, ds.
\end{equation}
As we can see in the next lemma, this integral allows us to estimate the sum of reciprocal norms of ideals that we mentioned in the previous section. 
\begin{lemma}\label{upperI}
We have
\[
I\leq \frac{1}{6}\ x^{1-\beta}\sum_{N(\ia)\leq x} \dfrac{1}{N(\ia)}.
\]
\end{lemma}

Before proving this lemma, let us first recall Perron's Formula, see \cite[p.243]{Tom} for example: if $y$ is any positive real number and $c>0$, then we have
\begin{equation}\label{eq:Per}
\frac{1}{2\pi i}\int_{c-i\infty}^{c+i \infty} \frac{y^s}{s} ds =
\begin{cases}
1 & \text{ if }y>1,\\
\frac{1}{2}& \textit{ if }y=1,\\
0 & \text{ if }0<y<1,
\end{cases}
\end{equation}
where by $\int_{c-i\infty}^{c+i \infty}$, we mean  $\underset{T\to \infty}{\textup{lim }}\ \int_{c-iT}^{c+i T}$.

\begin{proof}[Proof of Lemma~\ref{upperI}]
We begin by the following partial fraction decomposition: 
\begin{equation}\label{pc}
\frac{1}{s(s+2)(s+3)}=\frac{1}{6s}-\frac{1}{2(s+2)}+\frac{1}{3(s+3)}.
\end{equation}
Hence, by \eqref{eq:Per} we obtain 
\begin{equation}\label{eq:int_y}
\frac{1}{2\pi i}\int_{2-i\infty}^{2+i \infty}
\frac{y^s}{s(s+2)(s+3)}=
\begin{cases}
0 & \quad\text{ if }0<y<1,\\
\frac{1}{6} - \frac{y^{-2}}{2}+\frac{y^{-3}}{3} & \quad\text{ if }y\geq 1.
\end{cases}
\end{equation}
Since
\[
\zeta_{-d}(s)=\sum_{\ia} \dfrac{1}{N(\ia)^s}=\sum_{u^2a}\frac{\nu(a)}{u^{2s}a^s}
\]
converges absolutely for $\mathrm{Re(s)}>1$, we have 
\begin{align*}
I
& =  \frac{1}{2\pi i}\int_{2-i\infty}^{2+i\infty}
\sum_{\ia} \dfrac{1}{N(\ia)^{s+\beta}}\;\frac{x^s}{s(s+2)(s+3)}\ ds\\
& = \frac{1}{2\pi i}\int_{2-i\infty}^{2+i \infty}
\sum_{\ia} \dfrac{1}{N(\ia)^{\beta}}
\bigg(\frac{x}{N(\ia)}\bigg)^s\;\frac{1}{s(s+2)(s+3)}\ ds.
\end{align*}
Swapping summation and integration and  using \eqref{eq:int_y} (setting $y=\frac{x}{N(\ia)}$) yield
\begin{align*}
I 
& = \sum_{N(\ia)\leq x} \dfrac{1}{N(\ia)^{\beta}}
\bigg[\frac{1}{6} - \frac{N(\ia)^2}{2x^2}+\frac{N(\ia)^3}{3x^3}\bigg]\\
& \leq \frac{1}{6}\sum_{N(\ia)\leq x} \dfrac{1}{N(\ia)^{\beta}}\hspace*{1cm}\Big(\text{ since }\ \frac{1}{6}-\frac{1}{2y^2}+\frac{1}{3y^3}\leq \frac{1}{6} \ \text{ for any }\ y\geq 1\Big)\\
&\leq \frac{x^{1-\beta}}{6}\sum_{N(\ia)\leq x} \dfrac{1}{N(\ia)},
\end{align*}
which complete the proof of the lemma.
\end{proof}

\subsection{Lower bound on $I$}
By shifting the path of integration of the integral $I$ to $\mathrm{Re}(s)=-\beta$,  Equation \eqref{intI} can now be written as
\begin{equation}\label{splitintI}
I = \frac{L(1,\chi) x^{1-\beta}}{(1-\beta)(3-\beta)(4-\beta)} +\frac{1}{2\pi i}\int_{-\beta-i\infty}^{-\beta +i \infty} \zeta(s+\beta) L(s+\beta,\chi) \frac{x^s}{s(s+2)(s+3)} ds,
\end{equation}
where the first term on right-hand side comes from the simple pole of the integrand at $s=1-\beta$. Note that $s=0$ is also a singularity but it is removable since we assumed that $L(\beta,\chi)=0.$  Let us denote the integral on the right-hand side of \eqref{splitintI} by $J$, i.e., 
\[
J:=\frac{1}{2\pi i}\int_{-\beta-i\infty}^{-\beta+i\infty}\zeta(s+\beta)L(s+\beta,\chi)\frac{x^s}{s(s+2)(s+3)}\, ds.
\]
Then, one has
\[
|J|\leq \frac{x^{-\beta}}{2 \pi}\int_{-\infty}^{\infty}\frac{|\zeta(it)||L(it,\chi)|}{\sqrt{(\beta^2+t^2)((2-\beta)^2+t^2)((3-\beta)^2+t^2)}}\, dt.
\]

On the other hand, using our assumptions \eqref{eq:beta_ass} and $d> 300000000$, we deduce that 
\[
\beta^2> \left(1-6.5/\sqrt{300000000}\right)^2> 0.9996^2>0.999.
\]
Therefore,
\begin{equation}\label{eq:int_J}
|J|< \frac{x^{-\beta}}{2 \pi}\int_{-\infty}^{\infty}\frac{|\zeta(it)||L(it,\chi)|}{\sqrt{(0.999+t^2)(1+t^2)(4+t^2)}}\, dt.
\end{equation}

In order to find an upper bound for the above integral we need to obtain explicit bounds for  $|\zeta(it)|$ and $|L(it,\chi)|$.  The following explicit result can be found in \cite{BoundZ}: for $|t|\geq 3$, 
\begin{equation}\label{eq:bound_zeta}
|\zeta(1+it)|\leq \frac{3}{4}\log |t|.
\end{equation}
Similarly, for $|L(it,\chi)|$, Dudek in \cite{BoundL} obtained 
\begin{equation}\label{eq:bound_L}
|L(1+it,\chi)|\leq \log d +\log(e(|t|+14/5)).
\end{equation}
Using \eqref{eq:bound_zeta}, \eqref{eq:bound_L}, and the functional equations for the respective functions, we obtain the following lemma.
\begin{lemma}\label{lem:zeta_L}
For any real number $t$ such that $|t|\geq 3$, we have 
\begin{align*}
|\zeta(it)|& \leq \frac{3}{\sqrt{32\pi}}\sqrt{|t|}\log |t|, \ \text{ and }\\[.5em]
|L(it,\chi)| &\leq 0.4\ \sqrt{d|t|} \left(\log d +\log(e(|t|+14/5))\right).
\end{align*} 
\end{lemma} 
\begin{proof}
The functional equation of the Riemann zeta function gives
\[
|\zeta(it)|=\pi^{-1}\sinh(\pi|t|/2)|\Gamma(1-it)||\zeta(1-it)|.
\]
Since 
\begin{align*}
\sinh(\pi|t|/2)|\Gamma(1-it)|& =\sqrt{\frac{\pi}{2}|t|\tanh(\pi|t|/2)}\\
& =\sqrt{\frac{\pi}{2}|t|\left(1-\frac{2}{e^{\pi|t|}+1}\right)}\\
& \leq \sqrt{\frac{\pi}{2}|t|},
\end{align*}
we deduce from \eqref{eq:bound_zeta} that for $|t|\geq 3$ we have
\begin{align*}
|\zeta(it)|\leq \frac{3}{\sqrt{32\pi}}\sqrt{|t|}\log |t|.
\end{align*} 

Similarly the functional equation for $L(s,\chi)$ is as follows: if
\[
\Lambda(s,\chi)=\left(\frac{\pi}{d}\right)^{-(s+1)/2}\Gamma\left(\frac{s+1}{2}\right)L(s,\chi),
\]
then 
\[
\Lambda(1-s,\chi)=\frac{ik^{1/2}}{\tau(\chi)}\Lambda(s,\chi),
\]
where $\tau(\chi)=\sum_{k=1}^{d}\chi(k)\exp(2\pi i k/d)$ (here $\chi$ is real and $\chi(-1)=-1$). Using the fact that $|\tau(\chi)|=d^{1/2}$ and replacing $s$ by $it$  yield
\[
\left(\frac{\pi}{d}\right)^{-1} \left|\Gamma\left(\frac{2-it}{2}\right)\right||L(1-it,\chi)|=\left(\frac{\pi}{d}\right)^{-1/2} \left|\Gamma\left(\frac{1+it}{2}\right)\right||L(it,\chi)|.
\]
Hence
\[
|L(it,\chi)|=\left(\frac{d}{\pi}\right)^{1/2} \left|\Gamma\left(\frac{2-it}{2}\right)\right|\left|\Gamma\left(\frac{1+it}{2}\right)\right|^{-1}|L(1-it,\chi)|.
\]
Moreover, we have 
\[
\left|\Gamma\left(\frac{2-it}{2}\right)\right|=\sqrt{\frac{\pi|t|}{2\sinh(\pi |t|/2)}} \ \ \text{ and } \ \ \left|\Gamma\left(\frac{1+it}{2}\right)\right|=\sqrt{\frac{\pi}{\cosh(\pi |t|/2)}}.
\]
Thus, 
\[
\left|\Gamma\left(\frac{2-it}{2}\right)\right|\left|\Gamma\left(\frac{1+it}{2}\right)\right|^{-1}
 =\sqrt{\frac{|t|}{2}\coth(\pi|t|/2)}.
\]
Therefore, we deduce that
\begin{equation}
|L(it,\chi)|=\left(\frac{d}{\pi}\right)^{1/2} \sqrt{\tfrac{1}{2}|t|\coth(\pi|t|/2)} \, |L(1-it,\chi)|.
\end{equation}
If  $|t|\geq 3$, then $e^{\pi|t|}\geq e^{3\pi}> 12391$, so
\[
\sqrt{\tfrac{1}{2}|t|\coth(\pi|t|/2)} =\sqrt{\tfrac{|t|}{2}\left(1+\tfrac{2}{e^{\pi|t|}-1}\right)}  < \sqrt{\tfrac{1}{2}|t|\left(1+\tfrac{2}{12390}\right)}<0.708 \sqrt{|t|}.
\]
Thus for $|t|\geq 3$, we have 
\[
|L(it,\chi)|< \frac{0.708}{\sqrt{\pi}}\sqrt{d|t|}\, |L(1-it,\chi)|< 0.4 \sqrt{d|t|}\, |L(1-it,\chi)|.
\]
The proof of the lemma is complete by using \eqref{eq:bound_L} to estimate the right-hand side.
\end{proof}
Another consequence of the calculations in the proof above is that we also have 
\begin{equation}\label{eq:zetaL}
|\zeta(it)L(it,\chi)|=\frac{\sqrt{d}}{2\pi}|t\zeta(1-it)L(1-it,\chi)| \ \ \text{ for }\ \ t\in\mathbb{R}\setminus\{0\}.
\end{equation}
In view of \eqref{eq:int_J}, we consider the following integrals: 
\begin{align*}
J_1& :=\frac{1}{2\pi}\int_{-3}^{3}\frac{|t\zeta(1-it)|}{\sqrt{(0.999+t^2)((1+t^2)(4+t^2)}}\, dt, \\[1em]
J_2 & :=\frac{1}{2\pi} \int_{-3}^{3}\frac{|t\zeta(1-it)|\log(e(|t|+14/5))}{\sqrt{(0.999+t^2)((1+t^2)(4+t^2)}}\, dt, \\[1em]
J_3& :=\frac{0.6}{\sqrt{2\pi}}\int_3^{\infty}\frac{t\log t}{\sqrt{(0.999+t^2)((1+t^2)(4+t^2)}}\, dt, \\[1em]
J_4& :=\frac{0.6}{\sqrt{2\pi}}\int_3^{\infty}\frac{t\log t\log(e( t+14/5))}{\sqrt{(0.999+t^2)((1+t^2)(4+t^2)}}\, dt.
\end{align*}
By \eqref{eq:int_J}, \eqref{eq:zetaL} and Lemma~\ref{lem:zeta_L}, we have
\begin{equation}\label{eq:J2}
|J|\leq \frac{x^{-\beta}}{2\pi}\sqrt{d}\Big((J_1+J_3)\log d+J_2+J_4\Big).
\end{equation}

On the other hand, we can use a computer algebra system such as SageMath or Mathematica to calculate the $J_i$'s numerically. We obtained the following numerical values  (with high accuracy)  
\begin{align*}
J_1 &= 0.19692\ldots, \\
J_2 & = 0.45203\ldots, \\
J_3 & = 0.15661\ldots, \\
J_4 & =0.61360\ldots
\end{align*}
Rounding this values up at the $3$rd digit, and using \eqref{eq:J2}, we have 
\[
|J|\leq \frac{x^{-\beta}}{2\pi}\sqrt{d}\Big(0.354 \log d+1.067\Big) <\frac{x^{-\beta}}{2\pi}\Big(0.354+\frac{1.067}{\log d}\Big) \sqrt{d}\ \log d.
\]
Thus using $d\geq 300000000$ to estimate the term in brackets, we deduce that
\begin{equation}\label{estimateJ}
|J|< 0.066\ x^{-\beta}  \sqrt{d}\ \log d.
\end{equation}

Returning to the integral $I$. Recall from Equation \eqref{splitintI} that we have 
\[
I=\frac{L(1,\chi)x^{1-\beta}}{(1-\beta)(3-\beta)(4-\beta)}+J.
\]
Hence, using the estimate \eqref{estimateJ} for $J$ that we just achieved, we get
\[
I\geq \frac{x^{1-\beta}}{(1-\beta)}\left(\frac{L(1,\chi)}{(3-\beta)(4-\beta)}-0.066\ (1-\beta)\frac{\sqrt{d}\log d}{x}\right).
\] 
Since $x=\frac{1}{2}\sqrt{d}f(d)$, we deduce that
\[
I\geq \frac{x^{1-\beta}}{(1-\beta)}\left(\frac{L(1,\chi)}{(3-\beta)(4-\beta)}-0.132\ (1-\beta)\frac{\log d}{f(d)}\right).
\] 
In addition, by the class number formula for $d>4$, we have $L(1,\chi)=\frac{\pi h(-d)}{\sqrt{d}}$. So we finally get a lower estimate of $I$
\begin{equation}\label{eq:lowerI}
I\geq \frac{x^{1-\beta} }{(1-\beta)\sqrt{d}}\left(\frac{\pi h(-d)}{(3-\beta)(4-\beta)}-0.132\frac{(1-\beta)\sqrt{d}\log d}{f(d)}\right).
\end{equation}

\subsection{Upper bound on I}
We recall the bound from  Lemma \ref{upperI},  
\[
I \leq \frac{1}{6}\ x^{1-\beta}\sum_{N(\fa)\leq x} N(\fa)^{-1}.
\]
where $x=\tfrac{1}{2}\sqrt{d}f(d)$ and $f(d)\geq 1$ a function of $d$ to be chosen later.  Here we aim to estimate the sum on the right-hand side. For this, we choose another auxiliary function $\ell(d)$ that satisfies $f(d)\leq \ell(d).$ We have 
\begin{equation}\label{eq:N_pi_sqr}
\sum_{N(\fa)\leq x} N(\fa)^{-1}=\sum_{u^2a\leq x}\frac{\nu(a)}{u^2a}\leq \frac{\pi^2}{6}\sum_{a\leq x} \frac{\nu(a)}{a}.
\end{equation}
We split the last sum  into two parts
\[
\sum_{a\leq x} \frac{\nu(a)}{a}=\sum_{a\leq  \frac{1}{2}\sqrt{d}} \frac{\nu(a)}{a}+\sum_{\frac{1}{2}\sqrt{d}< a\leq x} \frac{\nu(a)}{a}=:S_0+S_1.
\] 
Furthermore, we also split $S_1$, 
\[
S_1=\sum\nolimits'\frac{\nu(a)}{a}+\sum\nolimits''\frac{\nu(a)}{a},
\]
where $\sum\nolimits'\frac{\nu(a)}{a}$ denotes the sum over all $a$ with $\frac{1}{2}\sqrt{d}< a\leq x$ such that $a$ has a prime divisor $p^{\alpha} > \ell(d).$ Hence, Lemma~\ref{lem:nu} yields 
\begin{equation}\label{eq:sum_prime}
\sum\nolimits'\frac{\nu(a)}{a}=\sum\nolimits'\frac{\nu(p^\alpha b)}{p^\alpha b}\leq \sum_{b< x/\ell(d)}\frac{\nu(b)}{b}\sum_{\Delta(b)< p^{\alpha}\leq  x/b} (2p^{-\alpha}),
\end{equation}
where
\[
\Delta(b):=\max\left\lbrace \ell(d), \tfrac{1}{2b}\sqrt{d}\right\rbrace.
\]
Recalling our notation from Section~\ref{sec-pre} that 
\[
\sum_{p^{\alpha}\leq y}p^{-\alpha}=\log\log y +B_2+\varepsilon(y).
\] 
So
\begin{align*}\sum_{\Delta(b)< p^{\alpha}\leq  x/b} (2p^{-\alpha})
& = 2\log\left(\frac{\log( x/b)}{\log\Delta(b)}\right)+2\Big(\varepsilon( x/b)-\varepsilon(\Delta(b))\Big)\\
& = 2\log\left(\frac{\log(f(d))+\log(\frac{1}{2b}\sqrt{d})}{\log\Delta(b)}\right)+2\Big(\varepsilon( x/b)-\varepsilon(\Delta(b))\Big)\\
&\leq 2\log\left(1+\frac{\log f(d)}{\log\Delta(b)}\right)+2\Big(\max_{ y\geq \ell(d)}\varepsilon(y)-\min_{y\geq \ell(d)} \varepsilon(y)\Big)\\
&\leq 2\log\left(1+\frac{\log f(d)}{\log \ell(d)}\right)+2\Big(\max_{ y\geq \ell(d)}\varepsilon(y)-\min_{y\geq \ell(d)} \varepsilon(y)\Big).
\end{align*}
Moreover, since $f(d)\leq \ell(d)$, we have $x/\ell(d)\leq \frac{1}{2}\sqrt{d}$. Thus
\[
\sum_{b< x/\ell(d)}\frac{\nu(b)}{b}\leq S_0.
\]
Therefore, we deduce from \eqref{eq:sum_prime}, that 
\begin{equation}
\sum\nolimits'\frac{\nu(a)}{a}\leq S_0\left(2\log\Big(1+\frac{\log f(d)}{\log \ell(d)}\Big)+2\Big(\max_{ y\geq \ell(d)}\varepsilon(y)-\min_{y\geq \ell(d)} \varepsilon(y)\Big)\right).
\end{equation}

As for the sum $\sum\nolimits''\frac{\nu(a)}{a}$, each positive integer $a$ contributing to this sum has no prime power divisor $>\ell(d)$. So the number of distinct prime divisors of such an $a$ is at least
\[
k_0:=\left\lceil \frac{\log(\frac{1}{2}\sqrt{d})}{\log \ell(d)}\right\rceil.
\]
The latter and the multiplicative property of the function $\nu(a)$ imply that 
\begin{align*}
\sum\nolimits''\frac{\nu(a)}{a}
& \leq \sum_{k\geq k_0}\frac{1}{k!}\left(\sum_{p^{\alpha}\leq \ell(d)} \frac{\nu(p^{\alpha})}{p^{\alpha}}\right)^k\\[.5em]
&\leq \frac{\sigma^{k_0}}{k_0!}\left(1+\frac{\sigma}{(k_0+1)}+\frac{\sigma^2}{(k_0+1)(k_0+2)}+\cdots\right),
\end{align*}
where
\[
\sigma:=\sum_{p^{\alpha}\leq \ell(d)} \frac{2}{p^{\alpha}}.
\]
If we choose $\ell(d)$ in such a way $k_0+1> \sigma$, then 
\[
1+\frac{\sigma}{(k_0+1)}+\frac{\sigma^2}{(k_0+1)(k_0+2)}+\cdots\leq 1+\frac{\sigma}{(k_0+1)}+\frac{\sigma^2}{(k_0+1)^2}+\cdots=\frac{1+k_0}{1+k_0-\sigma}.
\]
Hence, we obtain 
\begin{equation}
\sum\nolimits''\frac{\nu(a)}{a}\leq \frac{(1+k_0)\ \sigma^{k_0}}{(1+k_0-\sigma)\ k_0!}.
\end{equation}
Putting everything together, and using the result in Lemma~\ref{lem:h} that $S_0\leq \frac{h(-d)}{11}$, we arrive at the following estimate 
\begin{equation}\label{eq:N_up}
\sum_{N(\fa)\leq x} N(\fa)^{-1}\leq \frac{\pi^2}{66}h(-d)\left(1+2\log\Big(1+\frac{\log f(d)}{\log \ell(d)}\Big)+\mathrm{Er}(d,\ell(d))\right),
\end{equation}
where
\[
\mathrm{Er}(d,\ell(d)):=2\Big(\max_{ y\geq \ell(d)}\varepsilon(y)-\min_{y\geq \ell(d)} \varepsilon(y)\Big)+\frac{11\ (1+k_0)\ \sigma^{k_0}}{(1+k_0-\sigma)h(-d) \ k_0!}.
\]
This expression is not easy work with so let us derive a simpler bound for it. From Proposition \ref{prop}, we know that 
\[
\max_{ y\geq \ell(d)}\varepsilon(y)-\min_{y\geq \ell(d)} \varepsilon(y)\leq \frac{1.75}{(\log \ell(d))^2}+\min \left\lbrace\frac{0.2}{(\log \ell(d))^3},\ 10^{-4}\right\rbrace.
\]
If $\frac{0.2}{(\log \ell(d))^3}<10^{-4}$, then $\log \ell(d)>12$, so
\[
\frac{1.75}{(\log \ell(d))^2}+\frac{0.2}{(\log \ell(d))^3}\leq \frac{1}{(\log \ell(d))^2}\left(1.75+\frac{0.2}{\log \ell(d)}\right)<\frac{1.8}{(\log \ell(d))^2}.
\]
On the other hand if $\frac{0.2}{(\log \ell(d))^3}\geq 10^{-4}$, then $\log \ell(d)<13$, so 
\[
\frac{1.75}{(\log \ell(d))^2}+10^{-4}\leq \frac{1}{(\log \ell(d))^2}\left(1.75+\frac{(\log \ell(d))^2}{10000}\right)<\frac{1.8}{(\log \ell(d))^2}.
\]
Therefore, we get 
\begin{equation}
2\Big(\max_{ y\geq \ell(d)}\varepsilon(y)-\min_{y\geq \ell(d)} \varepsilon(y)\Big)< \frac{3.6}{(\log \ell(d))^2}.
\end{equation}
On the other hand, by Stirling's formula, we have 
\[
k_0!\geq \sqrt{2\pi k_0}\left(\frac{k_0}{e}\right)^{k_0}.
\]
Which implies that 
\[
\frac{\sigma^{k_0}}{k_0!}\leq \frac{1}{\sqrt{2\pi k_0}}\left(\frac{e\sigma}{k_0}\right)^{k_0}.
\]
All these together with $h(-d)\geq 101$, we obtain 
\begin{equation}\label{eq:Err}
E(d,\ell(d)) \leq \frac{3.6}{(\log \ell(d))^2} + \frac{11\ (1+k_0)}{101(1+k_0-\sigma)}\frac{1}{\sqrt{2\pi k_0}}\left(\frac{e\sigma}{k_0}\right)^{k_0}.
\end{equation}
\subsection{Final steps}
The purpose here is to choose suitable values for $f(d)$ and $\ell(d)$, but before we do that, let us first list all the constraints (on $f(d)$ and $\ell(d)$) that we assumed earlier. For  $d> 300000000$, we require
\begin{itemize}
\item  $f(d)\geq 1$,  
\item $f(d)\leq \ell(d)$, and
\item $k_0+1> \sigma$ (both sides of the inequality depend on $\ell(d)$).
\end{itemize}

\subsection*{Case $\log(d)\leq 42$}
We choose $f(d)=1$ ($\ell(d)$ will not be needed here), then \eqref{eq:N_pi_sqr}, Lemma~\ref{lem:h} and Lemma~\ref{upperI}  yield 
\[
I \leq \frac{1}{6}\ x^{1-\beta}\sum_{N(\fa)\leq \frac{1}{2}\sqrt{d}} N(\fa)^{-1}\leq \frac{\pi^2}{396} x^{1-\beta}h(-d).
\]
This and the lower bound of $I$ in \eqref{eq:lowerI} imply 
\[
\frac{1}{(1-\beta)\sqrt{d}} \left(\frac{\pi }{(3-\beta)(4-\beta)}-0.132\frac{(1-\beta)\sqrt{d}\log d}{h(-d)}\right)\leq \frac{\pi^2}{396}.
\]
Recall our assumption in \eqref{eq:beta_ass} that  $1-\beta\leq \frac{6.5}{\sqrt{d}}$. Since $d> 300000000$, the latter implies that $\beta>0.999,$ so
\[
\frac{\pi}{(3-\beta)(4-\beta)}> 0.523.
\]
Using $h(-d)\geq 101$, and the assumption \eqref{eq:beta_ass} again (to estimate the term $(1-\beta)\sqrt{d}$ inside the brackets above), we obtain 
\begin{align*}
(1-\beta)\sqrt{d}
& > \frac{396}{\pi^2}\left(0.523-0.132\frac{6.5\log d}{101}\right)\\
& > 20.984-0.341 \log d.
\end{align*}
The latter is greater than $6.6$ (if $\log(d)\leq 42$), contradicting \eqref{eq:beta_ass}. 
\subsection*{Case $42<\log(d)\leq 100$}
Here, we choose $f(d)=\ell(d)=16.$ The combination \eqref{eq:lowerI}, \eqref{eq:N_up} and Lemma~\ref{upperI} gives
\begin{align*}
\frac{1}{(1-\beta)\sqrt{d}}& \left(\frac{\pi }{(3-\beta)(4-\beta)}-0.132\frac{(1-\beta)\sqrt{d}\log d}{101f(d)}\right)\\
& \leq \frac{\pi^2}{396}\left(1+2\log\Big(1+\frac{\log f(d)}{\log \ell(d)}\Big)+\mathrm{Er}(d,\ell(d))\right).
\end{align*}  
This implies that  
\[
(1-\beta)\sqrt{d}> \frac{20.984-0.341\frac{\log d}{f(d)}}{1+2\log\Big(1+\frac{\log f(d)}{\log \ell(d)}\Big)+\mathrm{Er}(d,\ell(d))}.
\]
The numerator is obtained in the same way as in the previous case. Since $\ell(d)=16$, we have
\[
\sigma= 2\sum_{p^{\alpha}\leq 16}p^{-\alpha}<3.786. 
\]
Moreover, from \eqref{eq:Err},  we get
\begin{align*}
E(d,\ell(d)) &\leq \frac{3.6}{(\log 16)^2} + \frac{11\ (1+k_0)}{101(1+k_0-3.786)}\frac{1}{\sqrt{2\pi k_0}}\left(\frac{10.3}{k_0}\right)^{k_0} \\
& < 0.469+0.044 \frac{(1+k_0)}{(1+k_0-3.786)}\frac{1}{\sqrt{ k_0}}\left(\frac{10.3}{k_0}\right)^{k_0}. 
\end{align*}
Similarly, since $f(d)=\ell(d)=16$,   
\[
20.984-0.341\frac{\log d}{f(d)} \geq 20.984-0.022 \log d,
\]
and 
\[
1+2\log\Big(1+\frac{\log f(d)}{\log \ell(d)}\Big)  =1+2\log(2)<2.387.  
\]
Thus, we finally obtain 
\begin{equation}\label{eq:case2_lb}
(1-\beta)\sqrt{d}> \frac{20.984-0.022 \log d}{2.856+0.044 \frac{(1+k_0)}{(1+k_0-3.786)}\frac{1}{\sqrt{ k_0}}\left(\frac{10.3}{k_0}\right)^{k_0}}
\end{equation}
where, here 
\[
k_0=\left\lceil \frac{\log d-\log 4}{2\log 16 }\right\rceil.
\]
The right-hand side of \eqref{eq:case2_lb} is still difficult to estimate manually, so we did this numerically and the result is shown in Figure~\ref{Fig2}. The corners in the graph correspond to the points where the of value $k_0$ changes from an integer to the next. The minimum occurs in the first corner where $k_0$ changes from $8$ to $9$ i.e. when $\log d$ is close to $16\log 16-\log  4 \approx 45.747$. At this point we still have $(1-\beta)\sqrt{d}>6.53$ (when $k_0\geq  11$ the corners become less apparent because the second term in the denominator contributes very little). So, it is clear that we also obtain $(1-\beta)\sqrt{d}>6.5$ for all $d$ such that $42<\log(d)\leq 100$, which contradicts \eqref{eq:beta_ass}.  
\begin{figure}[h]
\centering
\includegraphics[scale=.6]{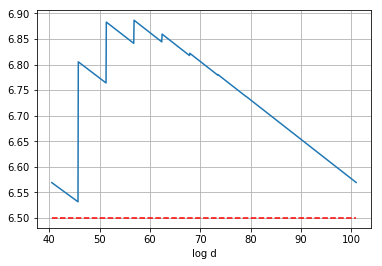}
\caption{Numerical plot for the case $42<\log(d)\leq 100$}\label{Fig2}
\centering
\end{figure}
\subsection*{Case $\log d> 100$}
Just as in the previous case, we also have the bound 
\begin{equation}\label{eq:case3_lb}
(1-\beta)\sqrt{d}> \frac{20.984-0.341\frac{\log d}{f(d)}}{1+2\log\Big(1+\frac{\log f(d)}{\log \ell(d)}\Big)+\mathrm{Er}(d,\ell(d))},
\end{equation} 
but we choose $f(d)=\ell(d)=0.5\ \log(\frac{1}{2}\sqrt{d}),$ which we simply abbreviate as$~t$ to make the reading easier. So from here, we will write everything in terms of $t$. The condition $\log d> 100$ implies that $t>24.65.$ Let us begin by estimating the terms in the denominator of \eqref{eq:case3_lb}. We have 
\[
k_0=\left\lceil\frac{2t}{\log t}\right\rceil \geq \frac{2t}{\log t},
\]
and since the right-hand side is at least $15.3$, we may assume that $k_0\geq 16.$ By Proposition~\ref{prop}, we have 
\[
\sigma =2\sum_{p^{\alpha}\leq t}p^{-\alpha}\leq 2\log\log t + 2.07.
\]
Thus, we deduce that 
\[
\frac{e\sigma}{k_0}\leq \frac{e \log t \ (2 \log\log t + 2.07)}{2t}
\]
It is easy to verify that the term on the right-hand side is a decreasing function of $t$ when $t>24.65.$ Hence, we get
\[
\frac{e\sigma}{k_0}< \frac{e \log (24.65) \ (2 \log\log (24.65) + 2.07)}{49.3}<0.778.
\]
In particular, we verified that $1+k_0>\sigma.$ We also have 
\[
\frac{1+k_0}{1+k_0-\sigma}= \frac{1}{1-\frac{\sigma}{1+k_0}}<\frac{1}{1-\frac{0.778}{e}}<1.401.
\]
Using these numerical estimates and the fact that $k_0\geq 16$, we obtain 
\begin{equation}\label{eq:case3_b1}
\frac{11\ (1+k_0)}{101(1+k_0-\sigma)}\frac{1}{\sqrt{2\pi k_0}}\left(\frac{e\sigma}{k_0}\right)^{k_0}<0.016 \ (0.778)^{16}<0.0003.
\end{equation}
We can see that the contribution from this term is very small. Let us now look the remaining terms in the denominator of the right-hand side of \eqref{eq:case3_lb}. Since we have chosen $f(d)=\ell(d)=t$, we have
\begin{equation}\label{eq:case3_b2}
1+2\log\Big(1+\frac{\log f(d)}{\log \ell(d)}\Big)+\frac{3.6}{(\log \ell(d))^2}=1+2\log 2 +\frac{3.6}{(\log t)^2}<2.737.
\end{equation}
The numerical value is obtain by rounding up the value at $t=24.65.$ Combining \eqref{eq:case3_b1} and  \eqref{eq:case3_b2}, we finally get 
\[
(1-\beta)\sqrt{d}>\frac{20.984-0.341\frac{\log d}{t}}{2.738}>7.663-0.125\frac{4t+\log 4}{t}>7
\]
for $t> 24.65.$  Once again, we obtain a constant strictly greater than $6.5$ which bounds $(1-\beta)\sqrt{d}$ from below. Since we have shown that this is the case for all possible values of $\log d\geq \log(300000000)$ the proof of Theorem~\ref{thm:main} is complete.
\section{Proof of Theorem~\ref{thm:t2}}\label{sec-t2}
We begin by the following consequence Theorem 1 in  Goldfeld-Schinzel~\cite{Gold} :  if $\beta$ exists, then
\begin{equation}\label{eq:GS_as}
1-\beta \geq  \left(\frac{6}{\pi^2}+o(1)\right)\frac{L(1,\chi)}{\displaystyle \sum_{a\leq \frac{1}{4}\sqrt{d}}\tfrac{\nu(a)}{a}} \, \, \text{ as }\ \ d\to \infty.
\end{equation}
For now, we need know how to estimate sums of the form  $\sum_{a\leq x}\frac{\nu(a)}{a}$. We start with the following observation which we already used in the proof Lemma~\ref{lem:h}: for any $1\leq x\leq \frac{1}{2}\sqrt{d} $, we have 
\begin{equation}
\sum_{a\leq x}\frac{\nu(a)}{a} \leq \sum_{a\leq y}\frac{2^{w(a)}}{a} \ \ \text{whenever }\ \ \sum_{a\leq y}2^{w(a)}\geq h(-d).
\end{equation}
The next lemma gives asymptotic estimates of the sums involving $2^{w(a)}$. 
\begin{lemma}\label{asym2omega}
As $y\to\infty$, we have
\[
\sum_{n\leq y}2^{w(n)}=\frac{6}{\pi^2}\ y\log y + O(y) \ \ \text{ and }\ \ 
\sum_{n\leq y}\frac{2^{w(n)}}{n}=\frac{3}{\pi^2} (\log y)^2 + O(\log y). 
\]
\end{lemma}
\begin{proof}
For each $n\geq 1$, the number $2^{w(n)}$ is equal to the number of squarefree divisors of $n$, i.e.,   
\[
2^{w(n)}=\sum_{d|n} |\mu(d)|.
\] 
Therefore,
\begin{align}
\sum_{n\leq y}2^{w(n)} 
& = \sum_{n\leq y} \sum_{d|n} |\mu(d)| \nonumber \\
& = \sum_{d\leq y}|\mu(d)|\sum_{q\leq \frac{y}{d}} 1 \nonumber \\
& = y \sum_{d\leq y}\frac{|\mu(d)|}{d} +O(y). \label{eq:sum2^w}
\end{align}
To estimate the sum in the last line, we use a well known estimate for the counting function of squarefree integers
\begin{equation}\label{countSI}
\sum_{n\leq y}|\mu(n)|=\frac{6}{\pi^2}\ y +O(\sqrt{y}).
\end{equation}
This is not too difficult to prove, we can even find a proof with an explicit error term in \cite{Mobius}. Hence, applying Abel's identity, we have 
\begin{equation}\label{abel}
\sum_{n\leq y}\frac{|\mu(n)|}{n}=\frac{1}{y}\sum_{n\leq y}|\mu(n)|+\int_1^y \frac{1}{t^2}\left(\sum_{n\leq t}|\mu(n)|\right)\ dt.
\end{equation}
The first term is obviously bounded, and the second can be estimated using \eqref{countSI}. Thus 
\begin{equation}\label{2omega}
\sum_{n\leq y}\frac{|\mu(n)|}{n}= \frac{6}{\pi^2} \, \log y +O(1).
\end{equation}
The estimate of $\sum_{n\leq y}2^{w(n)}$ in the lemma now follows from \eqref{eq:sum2^w}. For the second estimate in the lemma, we use Abel's identity again
\[
\sum_{n\leq y}\frac{2^{w(n)}}{n}=\frac{1}{y}\sum_{n\leq y} 2^{w(n)}+\int_1^y \frac{1}{t^2}\left(\sum_{n\leq t} 2^{w(n)}\right)\ dt.
\]
Then, we use the first estimate in the lemma, that we just proved, to estimate both terms on the right-hand side, and we obtain
\[
\sum_{n\leq y}\frac{2^{w(n)}}{n}=\frac{3}{\pi^2} (\log y)^2+O(\log y),
\]
which completes the proof of the lemma.
\end{proof}
One can find explicit upper bounds of the sums in Lemma~\ref{asym2omega}  in \cite[Lemma 12]{Trudgian}. Lower bounds can also be achieved using the same proof provided in that paper. We are now ready to prove the asymptotic formula in Theorem~\ref{thm:t2}. 
\begin{proof}[Proof of Theorem~\ref{thm:t2}]
We choose a positive number $y = y(d)$ in such a way that 
\[
\sum_{a\leq y-1}2^{w(a)}< h(-d)\leq \sum_{a\leq y}2^{w(a)},
\]
then, by the first estimate in Lemma \ref{asym2omega}, we have $h(-d)= \frac{6}{\pi^2}\  y(\log y)+O(y)$. Thus, writing $y$ in terms of $h(-d)$, we obtain
\begin{equation}\label{eq:y_h}
y= \left(\frac{\pi^2}{6}+O\left(\frac{1}{\log h(-d)}\right)\right) \frac{h(-d)}{\log h(-d)}.
\end{equation}
Similarly, using the second estimate in Lemma \ref{asym2omega}, we have
\begin{align*}
\sum_{a\leq \frac{1}{4}\sqrt{d}}\frac{\nu(a)}{a} \leq \sum_{a\leq y}\frac{2^{w(a)}}{a} = \frac{3}{\pi^2} (\log y)^2+O(\log y).
\end{align*}
Then, once again, expressing the right-hand side in terms of $h(-d)$ using \eqref{eq:y_h} yields
\[
\sum_{a\leq \frac{1}{4}\sqrt{d}}\frac{\nu(a)}{a} \leq \left(\frac{3}{\pi^2}+O\left(\frac{\log\log h(-d)}{\log h(-d)}\right)\right) (\log h(-d))^2.
\]
Plugging this into \eqref{eq:GS_as} completes the proof.
\end{proof}

\section{Concluding remarks}\label{sec-conc}
About further improvements of Theorem \ref{thm:main},  one might be able to push the constant $6.5$ to about $7$ by carefully choosing the values of $f(d)$ and $\ell(d)$. Another idea is to replace the term  $s(s+2)(s+3)$ in the definition of the integral $I$ in \eqref{intI} with $s(s+a)(s+b)$, then choose  $a$ and $b$ that give the best result. We have tried this and found out that $s(s+2)(s+3)$ is already very close to optimal. Replacing it will either make an insignificant improvement on the final result or worsen it.  What could really make a difference is any improvement of the bound in Lemma \ref{lem:h}, with $h(-d)\geq 101$ we could only get the factor $11$. We do not know if one could do significantly better than that. 

\bibliography{biblio}

\Addresses
\end{document}